\documentclass[12pt]{amsart}
\usepackage{amssymb}
\usepackage{graphicx}
\usepackage{hyperref}
\usepackage[all]{xy}
\usepackage{a4wide}
\usepackage{color}
\fontsize{12}{18}\selectfont

\newtheorem{lemma}{Lemma}[section]
\newtheorem{theorem}[lemma]{Theorem}
\newtheorem{corollary}[lemma]{Corollary}
\newtheorem{proposition}[lemma]{Proposition}
\theoremstyle{remark}
\newtheorem{remark}[lemma]{Remark}

\newcommand{\bbF}{\mathbb{F}}
\newcommand{\Fgag}{{\bar{F}}}
\def\dgag{{\bar d}}
\def\qgag{{\bar q}}
\newcommand{\wgag}{{\bar w}}
\newcommand{\Tr}{{\rm Tr}}
\def\Aut{{\rm Aut}}
\newcommand{\gal}{\textnormal{Gal}}

\begin{document}
\title[Totally ramified extensions]{On totally ramified extensions of discrete valued fields}
\author{Lior Bary-Soroker}%
\author{Elad Paran}
\address{Einstein Institute of Mathematics Edmond J. Safra Campus, Givat Ram, The Hebrew University of Jerusalem, Jerusalem, 91904, Israel}%
\email{barylior@math.huji.ac.il}%

\address{School of Mathematical Sciences, Tel Aviv University, Ramat Aviv, Tel Aviv, 69978, Israel}
\email{paranela@post.tau.ac.il}%

\thanks{The first author was partially supported by the Lady Davis fellowship trust, and the second author was partially supported by the Israel Science Foundation (Grant No. 343/07)}


\date{\today}

\begin{abstract}
We give a simple characterization of the totally wild ramified valuations in a Galois extension of fields of characteristic $p$.  This criterion involves the valuations of Artin-Schreier cosets of the $\bbF_{p^r}^\times$-translation of a single element.
We apply the criterion to construct some interesting examples.
\end{abstract}

\maketitle

\section{Introduction}
Let $F/E$ be a Galois extension of fields of characteristic $p$ of degree $q$, a power of $p$.
This work gives a simple criterion that classifies the totally ramified discrete valuations  of $F/E$.

The classical case where $F/E$ is a $p$-extension, hence generated by a root of an Artin-Schreier polynomial $X^p-X - a$ with $a\in E$, is well known: a discrete valuation $v$ of $E$ totally ramifies in $F$ if and only if the maximum of the valuation in the coset $a+E^p-E$ is negative, i.e., $m_{a,v} = \max \{v(b) \mid b\in a + E^p-E\}<0$.
A standard Frattini argument reduces the general case to finitely many $p$-extensions, or in other words to a criterion with finitely many elements. More precisely, there exist $a_1, \ldots, a_n\in E$ such that $v$ totally ramifies in $F$ if and only if $m_{a_i,v}<0$ for all $i$ ($n$ being the minimal number of generators of the Frattini quotient).

The goal of this work is to simplify this criterion and show that there exists (a single) $a\in E \bbF_q$ such that $v$ totally ramifies in $F$ if and only if  $m_{\gamma a,v} < 0$, for all $\gamma\in  \bbF_q^\times$ (see Theorem~\ref{thm:main}).

We apply our criterion to construct somewhat surprising examples: Assume $\bbF_q\subseteq E$ and that $F/E$ is generated by a degree $q$ Artin-Schreier polynomial $X^q- X -a$, $a\in E$. For a discrete valuation $v$ of $E$ let $M_{a,v} = \max \{ v(b) \mid b \in a + E^q - E\}$ be the maximum of the valuation of the $q$-Artin-Schreier  coset of $a$. It is an easy exercise to show that if $M_{a,v}<0$ and $\gcd(p,M_{a,v})=1$, then $v$ totally ramifies in $F$. So one might suspect that  $M_{a,v}$ encodes the information whether $v$ totally ramifies in $F$ as in the case $q=p$.  However this is false: We construct two extensions with the same $M_{a,v}<0$. In the first example $v$ totally ramifies in $F$ although $p\mid M_{a,v}$. In the second example $v$ does not totally ramify although it does ramify in $F$.

\textbf{Notation.}
Let $F/E$ be a Galois extension of fields of characteristic $p$ of degree a power of $p$. We write $q=p^r$ for the degree $[F:E] $ of the extension.
The symbol $v$ denotes a discrete valuation of $E$, and $w$ a valuation of $F$ lying above $v$.  We denote by $\bbF_{p^r}$ the finite field with $p^r$ elements. Sometimes we identify $\bbF_{p^r}$ with its additive group. The multiplicative group of a field $K$ is denoted by $K^\times$.

For an element $a\in E$ and discrete valuation $v$ of $E$ we denote
\begin{equation}
\label{def:max}
m_{a,v} = m(a,E,v) = \max\{ v(b) \mid b\in a + E^p - E\}\\
\end{equation}
if the valuation set of the elements in the coset is bounded, and $m_{a,v}=\infty$ otherwise.

\section{Classical Theory}

Let us start this discussion by recalling the well known case $q=p$. In this case Artin-Schreier theory tells us that $F = E(\alpha)$, where $\alpha$ satisfies an equation $X^p - X = a$, for some $a\in E$. Furthermore, one can replace $\alpha$ with a solution of $X^p - X = b$, for any $b\in a + E^p - E$.
We have the following classical result (cf.\ \cite[Proposition III.7.8]{Stichtenoth1993}).

\begin{theorem}
\label{thm:p-ext}
Assume $F = E(\alpha)$, for some $\alpha \in F$ satisfying an equation $X^p-X=a$, $a\in E$. Then  the following conditions are equivalent for a discrete valuation $v$ of $E$.
\begin{enumerate}
\item $v$ totally ramifies in $F$.
\item there exists $b \in a+E^p-E$ such that $\gcd(p, v(b)) = 1$ and $v(b)<0$.
\item $m_{a,v} < 0$.
\end{enumerate}
If these conditions hold, then $v(b) = m_{a,v}$, and in particular $v(b)$ is independent of the choice of $b$. Moreover, if $\beta$ is another Artin-Schreier generator, i.e., $F=E(\beta)$, and $\beta^p-\beta = a_\beta\in E$, then $m_{a_\beta,v} = m_{a,v}$.
\end{theorem}

We return to the case of an arbitrary $q =p^r$. Then a standard Frattini argument reduces the question of when a discrete valuation $v$ of $E$ totally ramifies in $F$ to extensions with $p$-elementary Galois group. Here a group $G$ is $p$-elementary if $G$ is abelian and of exponent $p$; equivalently $G\cong \bbF_q$. For the sake of completeness, we provide a formal proof of the reduction.

\begin{proposition}
\label{prop:red-p-elem}
There exists $\Fgag\subseteq F$ such that $\gal(\Fgag/E)$ is $p$-elementary and a discrete valuation $v$ of $E$ totally ramifies in $F$ if and only if $v$ totally ramifies in $\Fgag$.
\end{proposition}

\begin{proof}
Prolong $v$ to a valuation $w$ of $F$.
Let $G = \gal(F/E)$, let $\Phi = \Phi(G) = G^p [G,G]$ be the Frattini subgroup of $G$, and let $\Fgag = F^\Phi$ be the fixed field of $\Phi$ in $F$. Let $\wgag$ be the restriction of $w$ to $\Fgag$.
Then $\gal(\Fgag/E) \cong G/\Phi$ is $p$-elementary.

Let $I_{w/v}$, $I_{\wgag/v}$ be the inertia groups of $w/v$, $\wgag/v$, respectively. Let  $r\colon \gal(F/E)\to \gal(\Fgag/E)$ be the restriction map. Then $r(I_{w/v})=I_{\wgag/v}$ \cite[Proposition~I.8.22]{Serre1979}. This implies that $I_{w/v} = G$ if and only if $I_{\wgag/v} = r(I_{w/v}) = \gal(\Fgag/E)$  (recall that a subgroup $H$ of a finite group $G$ satisfies $H\Phi(G) = G$ if and only if $H=G$).
\end{proof}

\begin{remark}
The Frattini subgroup is the intersection of all maximal subgroups. Therefore $\Fgag$, as its fixed field, is the compositum of all minimal sub-extensions of $F/E$.
\end{remark}

Applying Theorem~\ref{thm:p-ext} for $\Fgag$ gives the following

\begin{corollary}\label{cor:finite_set}
Let $F/E$ be a Galois extension of degree $q=p^r$. Then there exist $a_1, \leq, a_n\in E$ such that for any discrete valuation $v$ of $E$ we have $v$ totally ramifies in $F$ if and only if $m_{a_i,v} < 0$ for all $i$. 
\end{corollary}

\section{Criterion for total ramification using one element}
In this section we strengthen Corollary~\ref{cor:finite_set} and prove that it suffices to take $\bbF_q^\times$-translation of a single element.
For this we need the following lemma.

\begin{lemma}
\label{lem:p-elementary-extensions}
Let $p$ be a prime and $q=p^r$ a power of $p$.
Consider a tower of extensions $\bbF_q\subset E \subseteq F$ with $q=[F:E]$. Assume $F=E(x)$ for some $x\in F$ that satisfies $a:=x^{q}-x \in E$.
Then the family of fields generated over $E$ by roots of $X^p - X - \gamma a$, where $\gamma$ runs over $\bbF^\times_q$ coincides with the family of all minimal sub-extensions of $F/E$.
\end{lemma}

\begin{proof}
Since $X^q - X - a = \prod_{\alpha\in \bbF_q} (X- (x+\alpha))$, the extension $F/E$ is Galois. Let $G = \gal(F/E)$, then the map
\[
\phi \colon \left\{
\begin{array}{ccc} G& \to& \bbF_q\\
\sigma &\mapsto& \sigma(x) - x
\end{array}\right.
\]
is well defined. Moreover it is immediate to verify that $\phi$ is an isomorphism.

Let $C$ be the kernel of the trace map $\Tr\colon \bbF_q\to \bbF_p$; $\Tr(u) =u^{p^{r-1}} + \cdots + u$. It is well known that $T$ is a non-trivial linear transformation \cite[Theorem~VI.5.2]{Lang2002} over $\bbF_p$. This implies that $T$ is surjective, so $C$ is a hyper-space of $\bbF_q$ (as a vector space over $\bbF_p$).

The minimal sub-extensions of $F/E$ are the fixed fields of maximal subgroups of $\gal(F/E)$, which correspond to hyper-spaces of $\bbF_q$ via $\phi$.
Let $C'$ be a hyper-space in $\bbF_q$. Then there exists an automorphism $M\colon \bbF_q \to \bbF_q$ under which $M(C')  = C$. But $\Aut (\bbF_q) = \bbF_q^\times$, so $M$ acts by multiplying by some $\gamma \in \bbF_q^\times$. Hence $ \gamma C' = C$.
Vice-versa, if $\gamma\in \bbF_q^\times$, then $\gamma^{-1} C$ is a hyper-space.
Therefore, it suffices to show, for an arbitrary $\gamma\in \bbF_q^\times$, that the fixed field $F'$ of $H:=\phi^{-1}(\gamma^{-1} C)$ in $F$ is generated by a root of $X^p - X - \gamma a$.

Let $y = \gamma x$. Then $F = E(y)$ and
\[
y^q - y = \gamma^q x^q - \gamma x = \gamma(x^q-x) = \gamma a.
\]
Let $z = y^{p^{r-1}} + \cdots + y^p + y$.
Then $z \not\in E$, hence $[F:E(z)]\leq p^{r-1}$. We have
\[
z^p-z =  y^{p^{r}} + \cdots + y^{p^2} + y^p - ( y^{p^{r-1}} + \cdots + y^p + y)= y^q-y=\gamma a.
\]
Thus $[E(z):E]\leq p$, and we get $[F:E(z)]=p^{r-1}$. To complete the proof we need to show that $F'=E(z)$, so it suffices to show that $H$ fixes $z$. Indeed, let  $\sigma \in H = \phi^{-1}(\gamma^{-1} C)$. Then $\beta := \sigma(y) - y = \gamma (\sigma(x)-x) = \gamma \phi(\sigma)\in C$. We have
\begin{eqnarray*}
\sigma(z) -z
	&=&  \sigma(y^{p^{r-1}} + \cdots + y^p + y) -  (y^{p^{r-1}} + \cdots + y^p + y)\\
	&=&  (\sigma(y) - y)^{p^{r-1}} + \cdots + (\sigma(y) - y)^p + (\sigma(y) - y)\\
	&=& \beta^{p^{r-1}} + \cdots + \beta = \Tr(\beta) = 0,
\end{eqnarray*}
as needed.
\end{proof}

We are now ready for the main result that classifies totally ramified discrete valuations of Galois extensions in characteristic $p$.

\begin{theorem}
\label{thm:main}
Assume $F/E$ is a Galois extension of fields of characteristic $p$ of degree  a power of $p$ and with Galois group $G$. Let $d=d(G)$ be the minimal number of generators of $G$ and let $q = p^r$, for some $r\geq d$ (e.g., $q=[F:E]$).
Let $F' = F \bbF_q$ and $E' = E \bbF_q$. If $v$ is a valuation of $E$, we denote by $v'$ its (unique) extension to $E'$. Then there exists $a\in E'$ such that for every discrete valuation $v$ of $E$ the following is equivalent.
\begin{enumerate}
\item $v$ totally ramifies in $F$.
\label{cond:wil_ram_a}
\item $v'$ totally ramifies in $F'$.
\label{cond:wil_ram_b}
\item $m(\gamma a, E', v') < 0$, for every $\gamma \in \bbF_q^\times$.
\label{cond:wil_ram_c}
\item There exists $b_\gamma \in \gamma a + (E')^p - E'$ such that $\gcd(p,v'(b_\gamma))=1$ and $v'(b_\gamma)<0$, for every $\gamma\in \bbF_q^\times$.
\label{cond:wil_ram_d}
\end{enumerate}
\end{theorem}

\begin{remark}
In the above conditions (c) and (d) it suffices that $\gamma$ runs over representatives of $\bbF_q^\times/\bbF_p^\times$.
\end{remark}

\begin{proof}
Since finite fields admit only trivial valuations, we get that both $F'/F$ and $E'/E$ are unramified, so \eqref{cond:wil_ram_a} and \eqref{cond:wil_ram_b} are equivalent.
Theorem~\ref{thm:p-ext} implies that \eqref{cond:wil_ram_c} and \eqref{cond:wil_ram_d} are equivalent. So it remains to proof that \eqref{cond:wil_ram_b} and \eqref{cond:wil_ram_c} are equivalent.
For simplicity of notation, we replace $F,E$ with $F',E'$ and assume that $\bbF_q\subseteq E$.

Let $\Fgag\subseteq F$ be the extension given in Proposition~\ref{prop:red-p-elem}. Let $\dgag$ be the minimal number of generators of $\gal(\Fgag/E)$. Then $\qgag = p^\dgag = [\Fgag:E]$ and $\dgag\leq d$. By Proposition~\ref{prop:red-p-elem} we may replace $\Fgag$ with $F$, and assume that $\gal(F/E) \cong \bbF_q$.

By Artin-Schreier Theory $F = E(x)$, where $x$ satisfies the equation $x^q-x=a$, for some $a\in E$. Lemma~\ref{lem:p-elementary-extensions} implies that all the minimal sub-extensions of $F/E$ are generated by roots of $X^p-X-\gamma a$, where $\gamma$ runs over $\bbF_q^\times$.
Note that $v$ totally ramifies in $F$ if and only if $v$ totally ramifies in all the minimal sub-extensions of $F/E$ (since if the inertia group is not the whole group, it fixes some minimal sub-extension, so $v$ does not ramify in this sub-extension).
This finishes the proof, since by Theorem~\ref{thm:p-ext} $v$ totally ramifies in all the minimal sub-extensions of $F/E$ if and only if $m(\gamma a, E, v)<0$, for all $\gamma \in \bbF_q^\times$.
\end{proof}

\section{An application}
We come back to the case where $\bbF_q \subseteq E\subseteq F$, and $F/E$ is a Galois extension with Galois group isomorphic to $\bbF_q$.
By Artin-Schreier Theory $F = E(x)$, where $x\in F$ satisfies an equation $X^q - X = a$, for some $a\in E$. This $a$ can be replaced by any element of the coset $a + E^q - E$.
If there exists $b\in a + E^q-E$ such that $v(b)< 0$ and $\gcd(q,v(b)) = 1$, then $v$ totally ramifies in $F$. It is reasonable to suspect that the converse also holds, as in the case $q=p$. We bring two interesting examples. The first is a totally ramified extension such that there exists no $b$ as above. The other construction is of an extension which is not totally ramified, although Condition~\eqref{cond:wil_ram_c} of Theorem~\ref{thm:main} holds for $\gamma = 1$.

\begin{proposition}
Let $p$ be a prime, $d\geq 1$ prime to $p$, $q=p^r$, and let $E = \bbF_q(t)$.
Consider the $t$-adic valuation,  i.e., $v(t) = 1$. Let $\gamma \neq 1$ be an element of $\bbF_q$ with norm $1$ (w.r.t.\ the extension $\bbF_q/\bbF_p$).  Consider an element
\[
a(t) = \frac{1}{t^{dp}} - \frac{\gamma}{t^d} + f(t) \in E
\]
and let $F=E(x)$, where $x$ satisfies $x^q - x = a$.
Then
\begin{enumerate}
\item If $d>1$ and $f(t) = \frac{1}{t}$, then $\gal(F/E) \cong \bbF_q$, $v$ totally ramifies in $F$, but there is no $b\in a + E^q - E$ whose valuation is prime to $p$.
\item If $f(t)  = t$, then $\max\{ v(b) \mid b\in a+ E^q-E\}<0$ but $v$ does not totally ramify in $F$.
\end{enumerate}
\end{proposition}

\begin{proof}
Let $\delta \in \bbF_q^\times$.
For $\epsilon\in \bbF_q$ with $\epsilon^p=\delta$
we set
\begin{equation}
\label{eq:b_delta}
b_\delta(t)= \delta a(t) - (\frac{\epsilon}{t^d})^p + \frac{\epsilon}{t^d} = \frac{\epsilon - \delta \gamma}{t^{d}} + \delta f(t).
\end{equation}
Since $\gamma\neq 1$ has norm $1$, $\gamma = \frac{\delta_0}{\delta_0^p}$, for some $\delta_0 \in \bbF_q$ (Hilbert 90).

Take $f(t)=\frac{1}{t}$. Then $v(b(t))$ is either $-d$ if $\epsilon \neq \delta \gamma $ or $-1$ if $\epsilon =\gamma \delta$, so $p\nmid v(b_\delta) < 0$.
By Theorem~\ref{thm:main}, $v$ totally ramifies in $F$.

To this end assume there exists $b\in a + E^q - E$ with $p\nmid v(b) < 0$, and let $-m = v(b)$. By Lemma~\ref{lem:p-elementary-extensions} the minimal sub-extensions of $F/E$ are generated by roots of $X^p - X -\delta b$, where $\delta \in \bbF_q^\times$.  But $v(\delta b) = v(b)$, so $-d = m(b_{\delta_0}, E, v) = m(b,E,v) = m(b_1,E,v) =-1$ (Theorem~\ref{thm:p-ext}). This contradiction implies that such $b$ does not exists, as needed for (a).

For (b) assume that $f(t)=t$, so   $v(b_{\delta_0^p} ) = v(f(t)) =1 $ by \eqref{eq:b_delta}. So $v$ is not totally ramified in $F$ (Theorem~\ref{thm:main}).
Assume there was $b\in a +E^q-E$ with $v(b)\geq 0$. Then all the minimal sub-extensions $F'$ of $F/E$ were generated by $X^p-X - \delta b$, where $\delta \in \bbF_q^\times$. But $v(\delta b) = v(b) \geq 0$, so all $F'$ are unramified (Theorem~\ref{thm:p-ext}). This conclusion contradicts the fact that the extension generated by $X^p-X-b_{\delta_0^p}$ is ramified. So $\max\{ v(b) \mid b\in a+ E^q-E\}<0$, as needed.
\end{proof}

\bibliographystyle{amsplain}

\begin{thebibliography}{1}

\bibitem{Lang2002}
Serge Lang, \emph{Algebra}, third ed., Graduate Texts in Mathematics, vol. 211,
  Springer-Verlag, New York, 2002. 

\bibitem{Serre1979}
Jean-Pierre Serre, \emph{Local fields}, Graduate Texts in Mathematics, vol.~67,
  Springer-Verlag, New York, 1979, Translated from the French by Marvin Jay
  Greenberg.

\bibitem{Stichtenoth1993}
Henning Stichtenoth, \emph{Algebraic function fields and codes}, Universitext,
  Springer-Verlag, Berlin, 1993.

\end{thebibliography}

\end{document}